\documentclass[oneside,a4paper,10pt]{amsart}

\usepackage[a4paper,left=22mm,right=42mm]{geometry} 
\usepackage{amsmath}
\usepackage{amsthm}
\usepackage{amscd}
\usepackage{euscript}
\usepackage{amssymb}
\usepackage{graphicx}
\usepackage{overpic}
\usepackage{sidecap}
\usepackage{enumitem}
\setlist[enumerate,1]{label = (\roman*),ref = (\roman*)}
\usepackage{contour}
\contourlength{1pt}
\usepackage[markup=underlined,todonotes={textsize=footnotesize}]{changes}
\definechangesauthor[color=red]{CM}
\definechangesauthor[color=blue]{AP}
\definechangesauthor[color=green]{OK}

\makeatletter
\def\paragraph{\@startsection{paragraph}{4}%
  \z@\z@{-\fontdimen2\font}%
  {\normalfont\bfseries}}
\makeatother

\newcommand{\mmeq}[2]{\noindent\hspace{\parindent}\parbox[b]{1em}{\hfill$ #1$}\
$\displaystyle#2$}

\def\dash{\discretionary{-}{}{-}\penalty1000\hskip0pt}

\theoremstyle{plain}
\newtheorem{theorem}{Theorem}[section]
\newtheorem{lemma}[theorem]{Lemma}
\newtheorem{proposition}[theorem]{Proposition}

\theoremstyle{remark}
\newtheorem{remark}[theorem]{Remark}
\newtheorem{example}[theorem]{Example}

\theoremstyle{definition}
\newtheorem{definition}[theorem]{Definition}

\newcommand{\R}{\mathbb{R}}
\newcommand{\Z}{\mathbb{Z}}
\newcommand{\DD}{\Delta}
\newcommand{\dd}{{\rm d}}

\makeatletter\def\@captionfont{\normalfont\footnotesize}\makeatother

\title{Stressability of Semi-Discrete Frameworks}

\author{Oleg Karpenkov, Christian M\"uller, Anna Pratoussevitch}

\begin{document}

\begin{abstract}
  In this paper we study the stressability of semi-discrete frameworks in the
  plane which are generated by a discrete sequence of smooth curves. We
  characterize their stressability property by the existence of stresses
  fulfilling certain difference-differential equation. In particular, we define
  a semi-discrete height function which we use to generate liftings. Furthermore,
  we show a semi-discrete analogue of the Maxwell-Cremona lifting property which
  implies that the stressable semi-discrete frameworks in the plane are
  precisely the orthogonal projections of semi-discrete conjugate surfaces in
  3-space. Finally, we discuss geometric implications for frameworks with
  vanishing boundary forces and characterize the liftability of frameworks which
  only consist of two neighboring curves forming one strip.
\end{abstract}

\maketitle
\tableofcontents

\setlength{\parskip}{2mm}

\section{Introduction}

Our paper is about characterizing stressability of semi-discrete
frameworks. We will show an analogous result to the classical Maxtwell-Cremona
lifting property: A planar semi\dash discrete framework admits a non-trivial force
load in static equilibrium if and only if it is the projection of a semi\dash
discrete surface with developable strips.

A planar bar-joint framework is stressable if non-vanishing forces
act on the bars by pushing or pulling while the framework remains in
equilibrium. This property is related to polyhedra in 3-space. 
J.C.~Maxwell and L.~Cremona have found the remarkable characterization
that a realization of a planar graph as a bar-joint framework in $\R^2$ is
stressable if and only if the bars coincide with the projection of
the edges of a polyhedron in
$\R^3$~\cite{maxwell1864xlv,maxwell1870reciprocal,cremona1872figure}.
The corresponding polyhedron is therefore called a \emph{lifting of the
framework}. 
The idea of lifting was used to visualize the statics of a framework
(graphic statics) and is still applied in modern structural engineering
for architecture~\cite{Baker-McRobie-2025,mbmm}.
This property has been utilized in several application of which we would only mention one where it was used to solve the carpenter's rule conjecture~\cite{Con03}.
For general references about liftings in discrete and computational geometry, we refer the reader to~\cite[Chapter 61]{handbookDCG}, as well as~\cite{WAC} and~\cite{sitharam2018handbook}. 

Various generalizations of the lifting property have been developed over the years. 
C.~Borcea and I.~Streinu proved an analogue of the theorem of Maxwell and
Cremona for planar periodic frameworks \cite{borcea2014liftings,BorStr}.
A notion of stresses for higher dimensional frameworks -- in that case cell-complexes -- has been investigated by K.~Rybnikov~\cite{rybnikov1999stresses}.
Trivalent multidimensional
frameworks have been characterized in~\cite{karpenkov+2022} in different ways.
Geometric conditions for the realization of equilibrium states of
frameworks with higher co-dimension, like bar-joint frameworks in $\R^3$,
have been investigated in~\cite{karpenkov+2021} while their liftings have been studied with
topological methods
in~\cite{karpenkov2023differentialapproachmaxwellcremonaliftings}.
The liftings of non-planar frameworks in $\R^2$ to polyhedral surfaces have recently been investigated in~\cite{karpenkov2025liftingssurfacesplane}.

In this paper, we generalize the classical Maxwell-Cremona lifting to the case of so-called semi-discrete surfaces or semi-discrete frameworks
(see Definition~\ref{def:lifting}).
Semi-discrete surfaces have been employed in geometry processing~\cite{pottmann+2008},
studied in the theory of transformations~\cite{mueller+2013},
investigated within finite and infinite flexibility~\cite{karpenkov-2015}
and brought together with curvature notions~\cite{karpenkovWallner-2014,mueller-2015}.
Semi-discrete surfaces can be seen as a discrete sequence of smooth parametrized curves such that there is a correspondence between neighboring curves.
We can think of this correspondence as very thin threads connecting corresponding points on neighboring curves.
Two neighboring curves are considered a strip.
The semi-discrete analogue of a polyhedral surface consists of developable strips.

In our setting, instead of forces acting on edges (as in classical discrete frameworks) we consider forces acting tangentially to the smooth curves, as well as infinitesimal forces acting along the connections between corresponding points on neighboring curves.
In this model, the local equilibrium condition becomes a difference-differential equation (Equation~(\ref{eq:sdequilibrium})).
The main contribution of our paper, beyond establishing a sensible framework for stressable semi-discrete structures, is a geometric characterization: we prove that a planar semi-discrete framework admits a non-trivial force load in static equilibrium if and only if it is the projection of a semi-discrete surface composed of developable strips (see Theorem~\ref{thm:liftability}).

\paragraph{Organization of the paper}

In Section~\ref{sec:sds} we give the definition of semi-discrete surfaces
and frameworks and derive a sensible notion for a self-stress and a
semi-discrete equilibrium condition. 
In Section~\ref{sec:mcl} we define a height function -- a lifting function
-- over planar semi-discrete surfaces.
We will show that a semi-discrete framework together with its height function constitutes a semi-discrete
surface with developable strips if and only if the framework is stressable.
Finally, in Section~\ref{sec:stressability} we look for the conditions on a framework that must be satisfied to ensure the existence of a self-stress.

\begin{SCfigure}[2][t]
  \begin{overpic}[width=.5\textwidth]{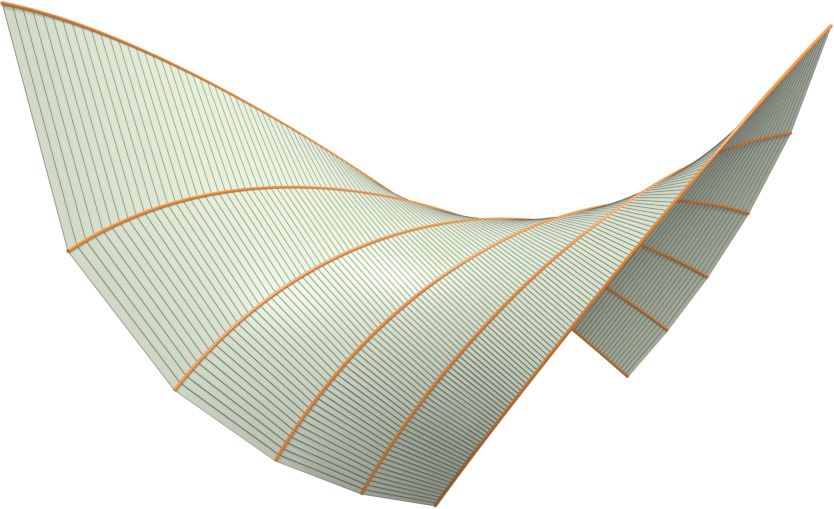}
    \put(0,28){$F_{i - 1}$}
    \put(17,11){$F_i$}
    \put(27,2){$F_{i + 1}$}
  \end{overpic}
  \caption{A semi-discrete surface consists of a discrete sequence of
  smooth curves $F_i$. The ``actual'' surface between two neighboring
  curves is in our case modeled by line segments
  between the corresponding points~$F_i(t)$ and~$F_{i+1}(t)$ for all~$t$.
  Hence the strips
  between two neighboring curves are ruled surfaces.
  The figure shows a semi-discrete surface made up of five strips. In this case, the strips are non-developable ruled surfaces, though this is not immediately apparent from the image.}
  \label{fig:general-sd-surface}
\end{SCfigure}

\section{Semi-Discrete Surfaces and Self-Stresses}
\label{sec:sds}

In this section we first recall the definition of semi-discrete surfaces and then derive a sensible notion for a semi-discrete framework to be stressable.
To do this we will follow the usual method to translate a physically ideal statical system to ODEs describing an equilibrium state.

\subsection{Semi-discrete surfaces}

Semi-discrete surfaces have a discrete and a smooth direction.
They can be interpreted as a sequence of smooth curves. 
The parameter domain is therefore a subset of $\Z \times \R$.
In what follows we will always assume the parameter domain to be a ``rectangle'' of the form 
\begin{equation*}
  U = \{0, \ldots, n\} \times [0, T]
\end{equation*}
with $n \in \Z_{\geq 0}$ and $T > 0$.

\begin{definition}
  A \emph{semi-discrete surface} is a mapping 
  \begin{align*}
    F : U &\to \R^d,\\
    (i, t) &\mapsto F(i, t).
  \end{align*}
  We often denote the smooth curves of the semi-discrete surface simply by $F_i(t) = F(i, t)$.
  For the \emph{semi-discrete partial derivatives} we use the common notation for the \emph{smooth derivative} of the curves and the \emph{discrete forward difference} of the polygons
  \begin{equation*}
    \partial_t F_i = \frac{\partial F_i}{\partial t} = \dot F_i 
    \quad\text{and}\quad 
    \DD F_i = F_{i + 1} - F_i.
  \end{equation*}
  A semi-discrete surface is called \emph{regular} if the discrete and
  smooth derivatives are linearly independent, i.e., the pairs
  \begin{equation}
    \label{eq:regular}
    (\dot F_i, \DD F_i)
    \qquad\text{and}\qquad
    (\dot F_i, \DD F_{i - 1})
  \end{equation}
  are linearly independent for all $i$ and $t$.
\end{definition}

We will only consider regular semi-discrete surfaces in this paper.

To align with the classical notion of a smooth surface we may think of a semi-discrete surface as being linearly interpolated between the curves (see Figure~\ref{fig:general-sd-surface}) with a parametrization defined piecewise in the form 
\begin{equation}
  \label{eq:strips}
    (u, v) \mapsto (1 + i - u) F_i(v) + (u - i) F_{i + 1}(v)
  \quad\text{when}\quad 
  u \in [i, i + 1),\ v \in [0, T].
\end{equation}
For each $i$, Equation~\eqref{eq:strips} describes a ruled surface. This
ruled surface can be developable or non-developable. This leads us to the
following subclass of so-called conjugate semi-discrete surfaces which are
particularly important for this paper.
\begin{definition}
  \label{def:conjugate}
  A semi-discrete surface is called \emph{conjugate} if for all $i$ the ruled surface in Equation~\eqref{eq:strips} is a developable ruled surface.
\end{definition}

\begin{remark}
  \label{rem:conjugate}
  A semi-discrete surface is conjugate if and only if 
  \begin{equation*}
    \DD F_i(t), \dot F_i(t), \dot F_{i + 1}(t)
  \end{equation*}
  are linearly dependent for all $i$ and $t$. This holds because
  developable surfaces are ruled surfaces which have a unique tangent
  plane along the regular points of each ruling.
\end{remark}

\begin{figure}[h]
  \hfill
  \begin{overpic}[width=.3\textwidth]{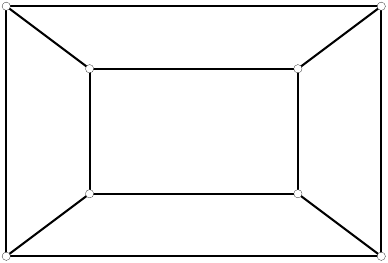}
  \end{overpic}
  \hfill
  \begin{overpic}[width=.3\textwidth]{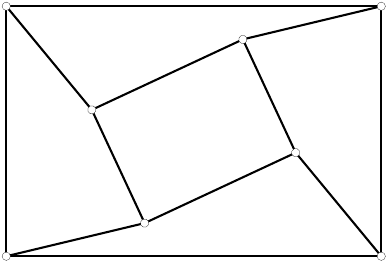}
  \end{overpic}
  \hfill{}
  \caption{Two discrete bar-joint frameworks. Whether a discrete framework
  is stressable or not can depend on the geometric realization. 
  The framework on the \emph{left} is stressable whereas the framework on
  the \emph{right} is \emph{not} stressable (see,
  e.g.,~\cite{karpenkov+2021}).}
  \label{fig:stressable-framewok}
\end{figure}

\subsection{Discrete frameworks}
\label{subsec:discrete-fw}

Before we consider stresses related to semi-discrete frameworks let us
quickly review the corresponding theory on finite graphs. Let $G = (V, E)$
be a finite graph with $V \subset \R^d$ as the vertex set and $E \subset V
\times V$ as the edge set.
We do not allow loops nor multiple edges. 
A realization of such a graph in $\R^d$ where the edges are realized as line segments is called a \emph{framework}.
For a detailed introduction to
frameworks see, e.g.,~\cite{connelly-guest-2022}.

A \emph{stress} is a map $\omega : E \to \R$. For each edge $p_i p_j \in
E$ the product $\omega_{ij} (p_j - p_i)$ is the force that acts on that
edge. The framework is in \emph{(static) equilibrium} if around each
vertex $p_i$ all forces add up to zero
\begin{equation} 
  \label{eq:discr-equilibrium}
  \sum_{\{j\mid p_ip_j\in E\}} \omega_{ij} (p_j - p_i) = 0.
\end{equation}
In that case the stress is called \emph{self-stress} and the framework
\emph{stressable}. Note that not all frameworks are stressable. See
Figure~\ref{fig:stressable-framewok} for examples of a stressable and a non-stressable
framework.

\subsection{Stresses for semi-discrete surfaces in the plane}

Since we investigate the stressability or equilibrium conditions of planar
semi-discrete surfaces ($d = 2$) we will call them simply frameworks. In
order to better distinguish frameworks and semi-discrete surfaces in
higher dimensions we will use the lower case letter $f$ for the
frameworks.

\begin{definition}
  \label{def:sdframework}
  A planar semi-discrete surface 
  \begin{align*}
    f : \{-1, 0, \ldots, n + 1\} \times [0, T] \to \R^2
  \end{align*}
  is called a \emph{$($semi-discrete$)$ framework}.
\end{definition}

Note that the indices of a semi-discrete framework start at $-1$ and end at $n + 1$.
We artificially add a curve at the beginning and at the end to be able to deal with boundary forces more easily.
We will see below two
types of forces acting at points of a semi-discrete framework.
One type acts along ``smooth'' derivative vectors $\dot f_i$ whereas the other one acts along ``discrete'' derivative vectors $\DD f_i$. Consequently, there will
be forces at the boundary curve~$f_0$ pointing in the direction~$f_{-1}-f_0$,
hence the advantage of introducing an extra curve~$f_{-1}$ and,
analogously, an extra curve~$f_{n + 1}$.

In analogy to discrete frameworks (see Section~\ref{subsec:discrete-fw}) we define stresses as two real-valued functions; one for the smooth curves and one for the discrete edges.

\begin{SCfigure}[2][tb]
  \begin{overpic}[width=.49\textwidth]{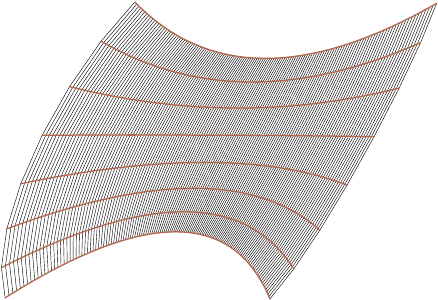}
    \put(85,35){\small$f_{i + 1}$}
    \put(80,24){\small$f_i$}
    \put(73,13){\small$f_{i - 1}$}
  \end{overpic}
  \caption{A semi-discrete framework is a planar semi\dash discrete surface
  $f : U \to \R^2$. It is therefore a semi\dash discrete parametrization of a
  part of the plane.}
  \label{fig:sd-framewok}
\end{SCfigure}

\begin{definition}
  Let $f$ be a semi-discrete framework (see Figure~\ref{fig:sd-framewok}).
  The pair $(\lambda, \mu)$ with 
  \begin{equation*}
    \lambda : \{0, \ldots, n\} \times [0, T] \to \R
    \quad\text{and}\quad
    \mu : \{-1, 0, \ldots, n\} \times [0, T] \to \R
  \end{equation*}
  is called a \emph{stress} on $f$;
  The function $\lambda_i$ is the stress for the curve~$f_i$ and $\mu_i$ the stress for the edges~$f_i f_{i + 1}$.
\end{definition}

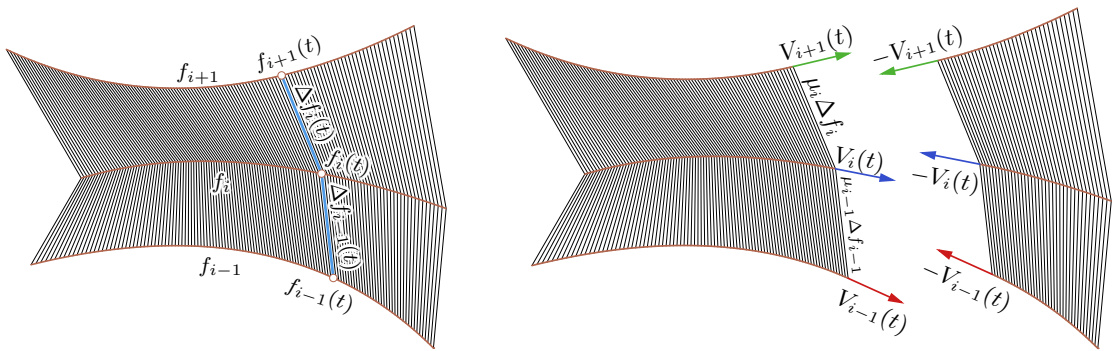
\begin{figure}[b]
  \begin{overpic}[width=.40\textwidth]{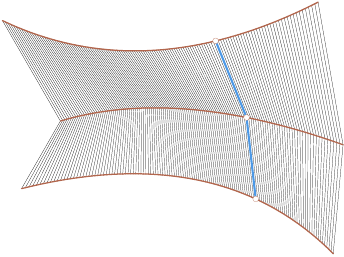}
    \put(39,62){\footnotesize$f_{i + 1}$}
    \put(56,64){\rotatebox{18}{\footnotesize$f_{i + 1}(t)$}}
    \put(47,37){\contour{white}{\footnotesize$f_i$}}
    \put(72,43){\rotatebox{-16}{\contour{white}{\footnotesize$f_i(t)$}}}
    \put(44,18){\footnotesize$f_{i - 1}$}
    \put(63,14){\rotatebox{-18}{\footnotesize$f_{i - 1}(t)$}}
    \put(64,60){\rotatebox{-64}{\contour{white}{\footnotesize$\DD f_i(t)$}}}
    \put(72,38){\rotatebox{-80}{\contour{white}{\footnotesize$\DD f_{i - 1}(t)$}}}
  \end{overpic}
  \hfill
  \begin{overpic}[width=.55\textwidth]{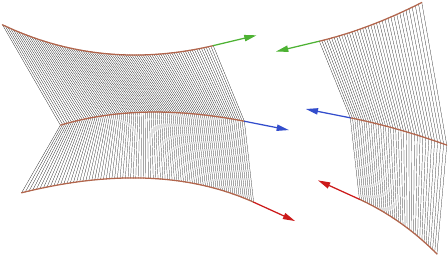}
    \put(54,31.3){\rotatebox{-10}{\small$V_i(t)$}}
    \put(66.5,28.3){\rotatebox{-10}{\small$-V_i(t)$}}
    \put(54,8){\rotatebox{-24}{\small$V_{i - 1}(t)$}}
    \put(68,13){\rotatebox{-24}{\small$-V_{i - 1}(t)$}}
    \put(45,48){\rotatebox{16}{\small$V_{i + 1}(t)$}}
    \put(60,47){\rotatebox{16}{\small$-V_{i + 1}(t)$}}
    \put(49,45){\rotatebox{-67}{\small$\mu_i \DD f_i$}}
    \put(55,29){\rotatebox{-83}{\tiny$\mu_{i - 1} \DD f_{i - 1}$}}
  \end{overpic}
  \caption{If we cut off a piece of a self-stressed semi-discrete framework~$f$ at parameter~$t$, we must replace the removed surface part by compensatory forces. For each curve~$f_i$, we need a force vector function $V_i=\lambda_i\dot f_i$.
  The magnitude of the force depends on the parameter~$t$.}
  \label{fig:forces}
\end{figure}

In a semi-discrete framework there are forces pulling or pushing along the
smooth curves $f_i$ and along the rulings. 
Suppose we cut a smooth curve $f_i$ at parameter $t$ and replace the force
generated by the removed curve $f_i(s)$ with $s > t$ by a force vector
$V_i(t)$. The force vector that holds back the removed curve is negative
to it: $-V_i(t)$ (see Figure~\ref{fig:forces}).
On the other hand the force vector is defined to be the stress times the tangent vector
\begin{equation}
  \label{eq:forcevector}
  V_i(t) = \lambda_i(t) \dot f_i(t).
\end{equation}
The force generated by each ruling is infinitesimally small and denoted by
\begin{equation*}
  \mu_i \DD f_i \, \dd t
  \quad\text{and}\quad 
  -\mu_{i - 1} \DD f_{i - 1} \, \dd t.
\end{equation*}
In analogy to the discrete equilibrium condition~\eqref{eq:discr-equilibrium}, we require that the sum of all forces around each point vanishes. Consequently, at a small segment $t, t + \DD t$ along the curve~$f_i$, the equilibrium condition is 
\begin{equation*}
  V_i(t + \DD t) - V_i(t) 
  + \int_t^{t + \DD t} 
  \big(\mu_i(s) \DD f_i(s) - \mu_{i - 1}(s) \DD f_{i - 1}(s)\big)\, \dd s
  = 0,
\end{equation*}
which implies in the limit
\begin{equation*}
  \dot V_i + \mu_i \DD f_i - \mu_{i - 1} \DD f_{i - 1} = 0.
\end{equation*}
Together with Equation~\eqref{eq:forcevector} we obtain
\begin{equation*}
  \dot \lambda_i \dot f_i + \lambda_i \ddot f_i + \mu_i \DD f_i 
  - \mu_{i - 1} \DD f_{i - 1} = 0.
\end{equation*}
We take this equation, derived with physical arguments, as the basis for our definition of stressable frameworks. 

\begin{definition}
  A stress $(\lambda, \mu)$ on a semi-discrete framework $f$ is a
  \emph{self-stress} if locally the difference\dash differential equation
  \begin{equation}
    \label{eq:sdequilibrium}
    \dot \lambda_i \dot f_i + \lambda_i \ddot f_i + \mu_i \DD f_i 
    - \mu_{i - 1} \DD f_{i - 1} = 0
  \end{equation}
  is satisfied for each $(i, t) \in U = \{0, \ldots, n\} \times [0, T]$.
  If a self-stress exists we call $f$ \emph{stressable}.
\end{definition}

\begin{proposition}
  For a stress $(\lambda, \mu)$ as above, 
  the \emph{complete force load}
  \begin{equation*}
    S_i(a, b) 
    = 
    \lambda_i(b) \dot f_i(b) - \lambda_i(a) \dot f_i(a) 
    + 
    \int_a^b (\mu_i(t) \DD f_i(t) - \mu_{i - 1}(t) \DD f_{i - 1}(t))\, \dd t
  \end{equation*}
  along $f_i$ between $a$ and $b$ vanishes for all $i$ and for all $a, b
  \in [0, T]$ if and only if $(\lambda, \mu)$ is a self-stress for $f$.
\end{proposition}
\begin{proof}
  Note that the derivative
  \begin{equation}
    \label{eq:sdb}
    \frac{\partial S_i(a, b)}{\partial b} 
    =
    \dot \lambda_i(b) \dot f_i(b) + \lambda_i(b) \ddot f_i(b) 
    +
    \mu_i(b) \DD f_i(b) - \mu_{i - 1}(b) \DD f_{i - 1}(b)
  \end{equation}
  is the left hand side of Equation~\eqref{eq:sdequilibrium}.
  If $S_i(a, b)$ vanishes for all $a, b$ then so does its
  derivative for all $b$, hence Equation~\eqref{eq:sdequilibrium} holds.

  On the other hand if Equation~\eqref{eq:sdequilibrium}
  holds then 
  \begin{equation*}
    \frac{\partial S_i(a, b)}{\partial b} = 0
  \end{equation*}
  for all $a$ and $b$. Consequently, for each fixed $a$, we have that 
  $S_i(a, b)$ is constant as a function of~$b$.
  Since $S_i(a, a) = 0$, we
  obtain that $S_i(a, b) = 0$ for all $a, b$ which concludes the proof.
\end{proof}

\section{Maxwell-Cremona liftings of semi-discrete stressed surfaces}
\label{sec:mcl}

In this section we establish the semi-discrete version of the
Maxwell-Cremona lifting. While in the classical setting stressable
frameworks are projections of polyhedral surfaces, in our semi-discrete
setting frameworks correspond to projections of semi-discrete conjugate
surfaces (Theorem~\ref{thm:liftability}).
In order to be able to formulate the precise statement we need to define a height function (Definition~\ref{defn:heightfunction}) in analogy with the discrete setting by adding/integrating along a path, in our case an increasing
semi-discrete path.

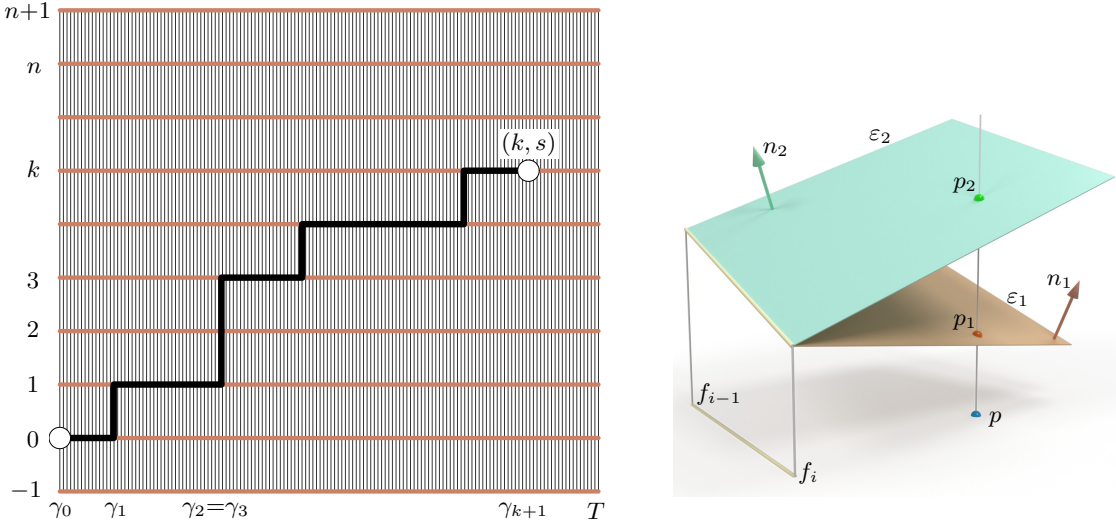
\begin{figure}[tb]
    \hfill
  \hfill
  \begin{overpic}[width=.5\textwidth]{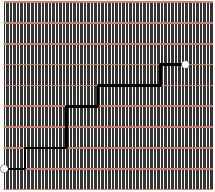}
    \put(0,-3){\small$\gamma_0$}
    \put(10,-3){\small$\gamma_1$}
    \put(24,-3){\small$\gamma_2{=}\gamma_3$}
    \put(81,-3){\small$\gamma_{k+1}$}
    \put(97,-4){\small$T$}
    \put(-7,0){\small$-1$}
    \put(-4,9){\small$0$}
    \put(-4,19){\small$1$}
    \put(-4,29){\small$2$}
    \put(-4,38){\small$3$}
    \put(-4,58){\small$k$}
    \put(-4,77){\small$n$}
    \put(-8,87){\small$n{+}1$}
    \put(81,63){\small\fboxsep1pt\colorbox{white}{$(k,s)$}}
  \end{overpic}
  \hfill
  \hfill
  \hfill
  \begin{overpic}[width=.40\textwidth]{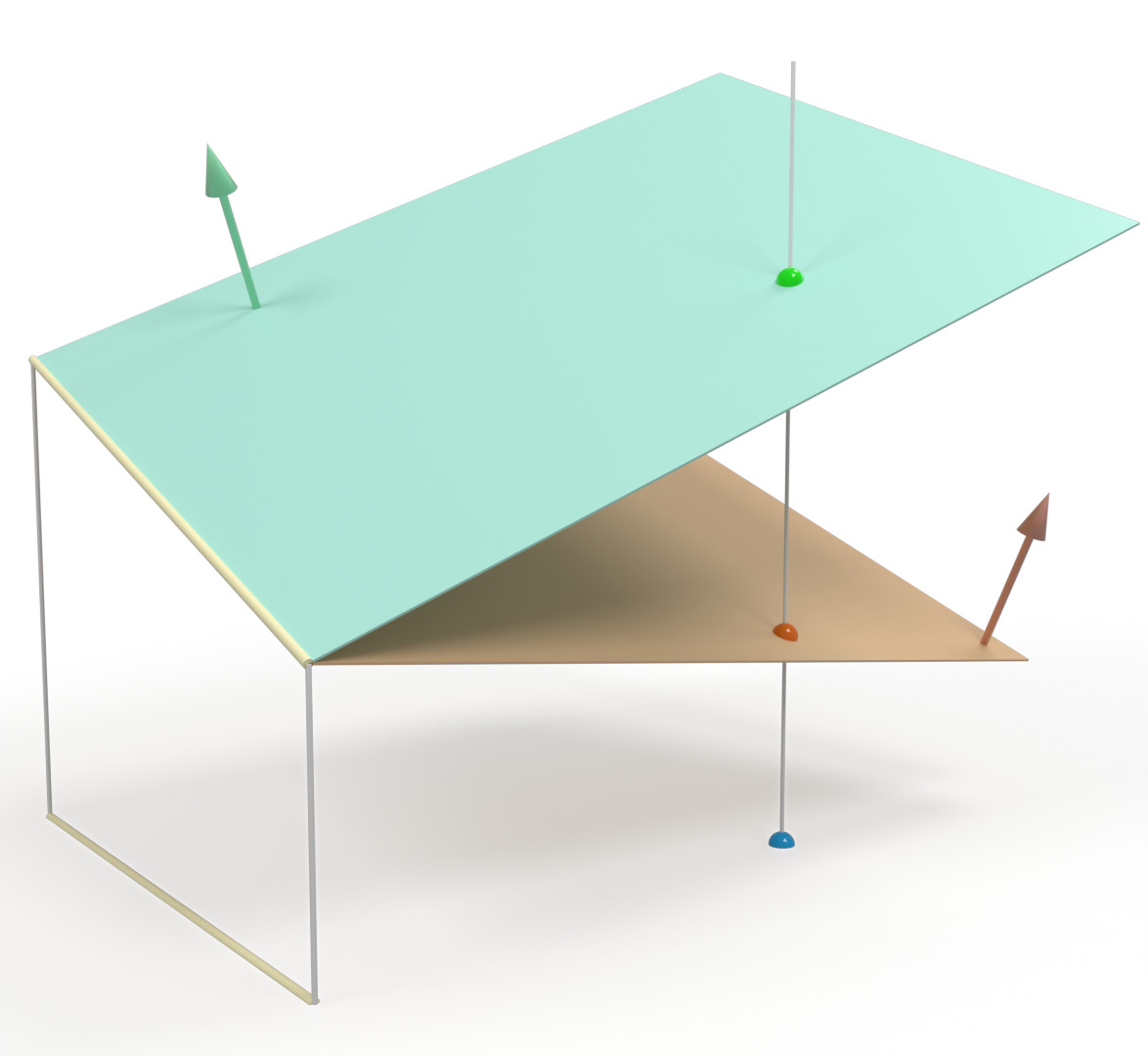}
    \put(5,22){\small$f_{i - 1}$}
    \put(28,4){\small$f_i$}
    \put(71,17){\small$p$}
    \put(84,48){\small$n_1$}
    \put(20,78){\small$n_2$}
    \put(63,39){\small$p_1$}
    \put(63,70){\small$p_2$}
    \put(75,44){\small$\varepsilon_1$}
    \put(44,81){\small$\varepsilon_2$}
  \end{overpic}
  \caption{\emph{Left:} An increasing semi-discrete path
  (Definition~\ref{defn:increasing}) in the semi-discrete parameter domain $U = \{0, \ldots, n\} \times [0, T]$.
  The horizontal lines correspond to smooth curves whereas the vertical lines correspond to discrete steps.
  The increasing semi-discrete path jumps at every $\gamma_i$ for $1 \leq i \leq k$.
  \emph{Right:} The classical Maxwell-Cremona lifting function adds up differences of heights of points projected onto edge-neighboring planes $\varepsilon_1, \varepsilon_2$.
  Their normal vectors are $n_1, n_2$.
  The height difference between $p_1$ and $p_2$ is 
  $\mu\det(\DD f_{i-1},p-f_i)$, where $\mu$ is the stress coefficient
  along the edge $f_{i - 1} f_i$ as explained before
  Definition~\ref{defn:heightfunction}.}
  \label{fig:path}
\end{figure}

\begin{definition}
  \label{defn:increasing}
  Let $0 \leq k \leq n$ and $s \in [0, T]$. 
  We say that a sequence $\gamma = (\gamma_i)_{i = 0}^{k + 1}$ with
  \begin{equation*}
    0 = \gamma_0 \leq \gamma_1 \leq \ldots \leq \gamma_{k + 1} = s
  \end{equation*}
  describes an \emph{increasing semi-discrete path} 
  as it determines a step function 
  \begin{equation*}
    t \mapsto (i, t)
    \quad\text{when}\quad 
    t \in [\gamma_i, \gamma_{i + 1}),
  \end{equation*}
  as illustrated in Figure~\ref{fig:path} (left).
\end{definition}

Note that if the interval $[\gamma_i, \gamma_{i + 1})$ is empty, the path
jumps over the smooth line $(i, t)$ as illustrated in
Figure~\ref{fig:path} (left) with $i = 2$.

In Definition~\ref{defn:heightfunction} we introduce our semi-discrete height function  that will be used to define the
lifting.
It is motivated by the classical lifting function which has the following property:
Consider an edge~$f_{i - 1} f_i$ of a discrete stressed framework in the plane~$\R^2$,
see Figure~\ref{fig:path}
(right) for an illustration. Let us call the stress coefficient $\mu$.
The discrete lifting is a polyhedral surface.
Let us denote its two planes which intersect in the edge that projects to $f_{i - 1} f_i$ by $\varepsilon_1,
\varepsilon_2$.
Their normal vectors are called $n_1, n_2$ and are
represented (or ``normalized'') by equal $z$-coordinats as follows:
$n_1 = (n_1^x, n_1^y, -1)$ and 
$n_2 = (n_2^x, n_2^y, -1)$. 
The classical Maxwell-Cremona lifting property also implies $n_1 - n_2 =
(\mu (f_{i - 1} - f_i)^\perp, 0)$ where ${}^\perp$ denotes the rotation
through 90 degrees (see, e.g.,~\cite{BorStr}).
Consider a point $p = (p^x, p^y, 0)$ and the two corresponding points
$p_1, p_2$ in the two planes $\varepsilon_1, \varepsilon_2$ which
vertically project to $p$.
The vertical heights of these two points over~$p$ are $z_j=n_j^x(p^x-f_i^x)+n_j^y(p^y-f_i^y)$ for $j=1,2$.
The difference in the height between $p_1$ and $p_2$ equals therefore
\begin{equation*}
  z_1 - z_2 
  = 
  \langle n_1 - n_2, (p - f_i, 0)\rangle 
  =
  \langle \mu (f_{i - 1} - f_i)^\perp, p - f_i\rangle
  =
  \mu \det(\DD f_{i - 1}, p - f_i).
\end{equation*}
The classical Maxwell-Cremona lifting function adds a term of the type above whenever a path in the projection crosses an edge $f_{i-1}f_i$ of the framework.
The following semi-discrete version is inspired by that formula. It
consists of a discrete part and a smooth part which are the limits of a
sum of such terms.

\begin{definition}
  \label{defn:heightfunction}
  Consider a framework $f$, a stress $(\lambda, \mu)$, a point $p = (k, s)
  \in U = \{0, \ldots, n\} \times [0, T]$ and an increasing semi-discrete path 
  $\gamma = (\gamma_i)_{i = 0}^{k + 1}$.
  We define a real-valued function that we call the \emph{height function}
  $H_\gamma$ as follows
  \begin{align*}
    H_\gamma(f(p)) 
    &= 
    \sum_{i = 0}^{k - 1} 
    \Big(\lambda_i(\gamma_{i + 1}) 
    \det(\dot f_i(\gamma_{i + 1}), f(p) - f_i(\gamma_{i + 1}))
    \Big)
    \\
    &-
    \sum_{i = 0}^k
    \int\limits_{\gamma_i}^{\gamma_{i + 1}} 
    \mu_{i - 1}(u) \det(\DD f_{i - 1}(u), f(p) - f_{i - 1}(u))\, \dd u.
  \end{align*}
\end{definition}

The height function~$H_\gamma$ is a map from~$U$ to~$\R$ but in general, the value
of $H_\gamma(f(p))$ depends on the choice of~$\gamma$.
We will show in Theorem~\ref{thm:hindependent} that the height function $H_\gamma(f(p))$ does \emph{not} depend on $\gamma$ if $(\lambda, \mu)$ is a self-stress for~$f$.
But first we will show the following slightly technical proposition which says that any increasing semi-discrete path~$\gamma$ can be
replaced by a simple ``L''-shaped path (see Figure~\ref{fig:simple-path}).

\begin{SCfigure}[2][tb]
  \begin{overpic}[width=.4\textwidth]{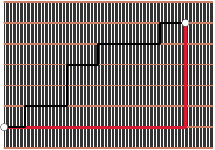}
    \put(49,45){\small\fboxsep1pt\colorbox{white}{$\gamma$}}
    \put(81,30){\small\fboxsep1pt\colorbox{white}{$\delta$}}
  \end{overpic}
  \caption{Two increasing semi-discrete paths from $(0,0)$ to $(k, s)$, a general path $\gamma$ and a simple ``L''-shaped path $\delta$.
  Computing the height function $H$ along both paths yields the same
  result if $(\lambda, \mu)$ is a self-stress for~$f$ (see
  Proposition~\ref{prop:pathindependence}).}
  \label{fig:simple-path}
\end{SCfigure}

\begin{proposition}
  \label{prop:pathindependence}
  Let $(\lambda, \mu)$ be a self-stress for a framework $f$ and let $p =
  (k, s)$.
  Let $\gamma,\delta$ be two increasing semi-discrete
  paths with 
  \begin{equation*}
    \delta_0 = \gamma_0,
    \qquad
    \delta_1 = \ldots = \delta_{k + 1} = \gamma_{k + 1} = s,
  \end{equation*}
  see Figure~\ref{fig:simple-path} for an illustration of such paths. Then
  $H_\gamma(f(p)) = H_\delta(f(p))$.
\end{proposition}  
\begin{proof}
  Let us assume $k = 1$. Then $\gamma = (\gamma_0, \gamma_1, \gamma_2)$
  and 
  $\delta = (\gamma_0, \gamma_2, \gamma_2)$. 
  Let us further consider the increasing semi-discrete path
  $\hat \gamma = (\gamma_0, t, \gamma_2)$ with $\gamma_0 \leq t\leq\gamma_2$.
  We set $D(t) = H_{\hat\gamma}(f(p)) - H_\delta(f(p))$. With this notation
  we have $D(\gamma_1) = H_\gamma(f(p)) - H_\delta(f(p))$. Consequently, if we can 
  show that $D = 0$, this would imply the proposition. We have
  \medskip

  \mmeq{}{
    H_{\hat\gamma}(f(p)) 
    =
    \lambda_0(t) \det(\dot f_0(t), f(p) - f_0(t))
  }\medskip

  \mmeq{}{
    -
    \int_{\gamma_0}^{t} 
    \mu_{-1}(u) \det(\DD f_{-1}(u), f(p) - f_{-1}(u))\, \dd u
    -
    \int_{t}^{\gamma_2} 
    \mu_0(u) \det(\DD f_0(u), f(p) - f_0(u))\, \dd u.
  }\medskip
  
  Using $\DD f_{-1}=f_0-f_{-1}$
  we obtain
  $\det(\DD f_{-1},f(p)-f_{-1})
  =\det(\DD f_{-1},f(p)-f_0)$.
  Differentiating~$D(t)$ yields
  \begin{equation*}
    \begin{aligned}
      \frac{\partial}{\partial t} D(t)
      =&
      \frac{\partial}{\partial t} H_{\hat\gamma}(f(p)) 
      - 
      \frac{\partial}{\partial t} H_\delta(f(p))
      =
      \frac{\partial}{\partial t} H_{\hat\gamma}(f(p)) 
      \\
      =&
      \dot \lambda_0(t) \det(\dot f_0(t), f(p) - f_0(t))
      +
      \lambda_0(t)\det(\ddot f_0(t),f(p)-f_0(t))
      \\
      &-
      \mu_{-1}(t)\det(\DD f_{-1}(t),f(p)-f_{-1}(t))
      +
      \mu_0(t)\det(\DD f_0(t),f(p)-f_0(t))\\
      =&\det\big(\dot \lambda_0(t)\dot f_0(t) +\lambda_0(t)\ddot f_0(t)+\mu_0(t)\DD f_0(t)-\mu_{-1}(t)\DD f_{-1}(t),f(p)-f_0(t)\big). 
    \end{aligned}
  \end{equation*}
  Equation~\eqref{eq:sdequilibrium} implies that one of the vectors in the determinant is zero, therefore
  \begin{equation}
    \label{eq:D}
    \frac{\partial}{\partial t} D(t)=0
  \end{equation}
  and $D$ is constant.
  For $t=\gamma_2$ we have $\delta=\hat\gamma$ and therefore 
  $D(\gamma_2)=0$. Consequently, $D =
  0$ and therefore $H_\gamma(f(p)) = H_\delta(f(p))$.
  Applying induction implies $H_\gamma(f(p)) = H_\delta(f(p))$ for $k > 1$.
\end{proof}

\begin{theorem}
  \label{thm:hindependent}
  Let $(\lambda, \mu)$ be a stress for a framework $f$.
  Then the following are equivalent:
  \begin{enumerate}
    \item\label{itm:1} $(\lambda, \mu)$ is a self-stress for $f$. 
    \item\label{itm:2} $H_\gamma(f(p))$ is independent of the choice of
      $\gamma$ for any $p \in U = \{0, \ldots, n\} \times [0, T]$.
    \item\label{itm:3} $H_\gamma(f(n, T))$ is independent of the choice of
      $\gamma$.
  \end{enumerate}
\end{theorem}
\begin{proof}
  The implication \ref{itm:1}$\Rightarrow$\ref{itm:2} follows directly
  from Proposition~\ref{prop:pathindependence} as all paths $\gamma$ yield the
  same value $H_\gamma(f(p))$ as the special path $\delta$ if $(\lambda,
  \mu)$ is a self-stress.

  As for the implication \ref{itm:2}$\Rightarrow$\ref{itm:1}, we consider
  the increasing semi-discrete path $\hat\gamma = (\gamma_0, t, \gamma_2)$.
  Since we assume that $H_\gamma$ does not depend on $\gamma$ we have that 
  $H_{\hat\gamma}(f(p))$ is constant as a function of~$t$ and therefore 
  $\frac{\partial}{\partial t} H_{\hat\gamma}(f(p))=0$.
  Using Equation~\eqref{eq:D}, we obtain
  \begin{equation}
    \label{eq:DD}
    \frac{\partial}{\partial t} H_{\hat\gamma}(f(p)) = 
    \det(\dot \lambda_0(t) \dot f_0(t) + \lambda_0(t) \ddot f_0(t) +
    \mu_0(t) \DD f_0(t) 
    - \mu_{-1}(t) \DD f_{-1}(t), f(p) - f_0(t)) 
    = 0.
  \end{equation}
  Furthermore, this equation also holds for any~$s$ and any $p=(1,s)$.
  We then have $f(p)=f(1,s)=f_1(s)$. 
  Differentiating Equation~\eqref{eq:DD} with respect to~$s$ and then substituting
  $t=s$ yields
  \begin{equation*}
    \det(\dot \lambda_0(s) \dot f_0(s) + \lambda_0(s) \ddot f_0(s) +
    \mu_0(s) \DD f_0(s) 
    - \mu_{-1}(s) \DD f_{-1}(s), \dot f_1(s)) 
    = 0.
  \end{equation*}
  Substituting $t = s$
  into Equation~\eqref{eq:DD} yields
  \begin{equation*}
    \det(\dot \lambda_0(s) \dot f_0(s) + \lambda_0(s) \ddot f_0(s) +
    \mu_0(s) \DD f_0(s) 
    - \mu_{-1}(s) \DD f_{-1}(s), \DD f_0(s)) 
    = 0.
  \end{equation*}
  Consequently, the determinant of the vector
  \begin{equation*}
  v=\dot\lambda_0(s)\dot f_0(s)+\lambda_0(s)\ddot f_0(s)+\mu_0(s)\DD f_0(s)-\mu_{-1}(s)\DD f_{-1}(s)
  \end{equation*}
  with each of the two linearly independent vecotors $\DD f_0(s),\dot f_1(s)$ vanishes and 
  therefore the vector~$v$ itself vanishes. 
  This implies the
  stressability condition~\eqref{eq:sdequilibrium}, hence
  $(\lambda, \mu)$ is a self-stress.
  
  The implication \ref{itm:2}$\Rightarrow$\ref{itm:3} follows from
  choosing $p = (n, T)$.
  As for the implication \ref{itm:3}$\Rightarrow$\ref{itm:2} let $\gamma$
  and $\varphi$ be two increasing semi-discrete paths ending at $p \in U$.
  We extend them both to $(n, T)$ by adding the same extra path and obtain
  the extended paths $\tilde \gamma$ and $\tilde \varphi$. Then
  by~\ref{itm:3} we have that $H_{\tilde \gamma}((n, T)) = H_{\tilde
  \varphi}((n, T))$. Expanding this equation using
  Definition~\ref{defn:heightfunction} and cancelling matching terms from
  the same extra path yields $H_{\gamma}(f(p)) = H_{\varphi}(f(p))$.
\end{proof}

We now thake this heigh function that does not depend on the path for
self-stressed frameworks and define a surface -- a lifting -- orthogonally
above the framework.

\begin{definition}
\label{def:lifting}
  Let $(\lambda,\mu)$ be a self-stress for a framework~$f$.
  Since in this case $H_\gamma(p)$ does not depend on $\gamma$,
  we can
  define the following function 
  \begin{align*}
    L : U &\to \R^2 \times \R\\
    p &\mapsto L(p) = (f(p), H_\gamma(f(p))).
  \end{align*}
  We call this function a \emph{$($semi-discrete$)$ lifting} of the self-stressed
  framework~$f$.
\end{definition}

The lifting of a curve~$f_k(t)$ of the self-stressed framework~$f$ is a space curve $L(k,t)=(f_k(t),H_\gamma(f(k,t)))$ for some increasing semi-discrete path~$\gamma$.
For the specific path~$\gamma$ with
\begin{equation*}
  \gamma_0=\ldots=\gamma_k=0 
  \quad\text{and}\quad
  \gamma_{k + 1}=t,
\end{equation*}
we obtain $L(k, t) = (f_k(t),H_\gamma(f(k,t)))$ with
\begin{equation}
\label{eq:Hk}
\begin{aligned}  
  &H_\gamma(f(k, t))=\\
  &=
  \sum_{i = 0}^{k - 1} 
  \Big(\lambda_i(0) 
  \det(\dot f_i(0), f_k(t) - f_i(0))
  \Big)
  -
  \int_0^t 
  \mu_{k - 1}(u) \det(\DD f_{k - 1}(u), f_k(t) - f_{k - 1}(u))\, \dd u.
\end{aligned}  
\end{equation}

\begin{proposition}
  \label{prop:affine}
  Let $f$ be a self-stressed framework with stress functions $(\lambda,
  \mu)$. Let us further consider a framework $g$ obtained by reversing the
  order, i.e., $g_k = f_{n - k}$, and with stress functions $(\hat \lambda,
  \hat \mu)$ where $\hat \lambda_k = -\lambda_{n - k}$ and $\hat \mu_k =
  -\mu_{n - k - 1}$.
  Then, the lifting of $g$ is the same as the lifting of $f$ up to adding
  an affine function to the height function $($and up to reversing the
  order$)$.
  The affine function reads
  \begin{equation*}
    x \mapsto 
    \sum_{\substack{i = 0}}^{n} 
    \Big(\lambda_i(0) 
    \det(\dot f_i(0), x - f_i(0))
    \Big).
  \end{equation*}
\end{proposition}

\begin{proof}
  Let us denote by $H^f$ and $H^g$ the height functions corresponding to
  frameworks $f$ and $g$, respectievely.
  Using Equation~\eqref{eq:Hk}, we can express the difference
  $H^f_\gamma(f(k, t)) - H^g_\gamma(g(n - k, t))$ as the sum of
  \begin{equation*}
    \sum_{i = 0}^{k - 1} 
    \Big(\lambda_i(0) 
    \det(\dot f_i(0), f_k(t) - f_i(0))
    \Big)
    -
    \sum_{i = 0}^{n - k - 1} 
    \Big(\hat \lambda_i(0) 
    \det(\dot g_i(0), g_{n - k}(t) - g_i(0))
    \Big)
  \end{equation*}
  and the integral part
  \begin{align*}
    &-\int_0^t 
    \mu_{k - 1}(u) \det(\DD f_{k - 1}(u), f_k(t) - f_{k - 1}(u))\, \dd u
    \\
    &+
    \int_0^t 
    \hat \mu_{n - k - 1}(u)
    \det(\DD g_{n - k - 1}(u), g_{n - k}(t) - g_{n - k - 1}(u))\, \dd u.
  \end{align*}
  We use $g_{n - k} = f_k$,
  $g_{n - k - 1} = f_{k + 1}$,
  $\DD g_{n - k - 1}=f_k-f_{k + 1}=-\DD f_k$
  to rewrite these two terms as
  \begin{equation}
  \label{eq:affinesum}
  \begin{aligned}
    &
    \sum_{i = 0}^{k - 1} 
    \Big(\lambda_i(0) 
    \det(\dot f_i(0), f_k(t) - f_i(0))
    \Big)
    +
    \sum_{i = 0}^{n - k - 1} 
    \Big(\lambda_{n - i}(0) 
    \det(\dot f_{n - i}(0), f_k(t) - f_{n - i}(0))
    \Big)
    \\
    &=
    \sum_{\substack{i = 0\\i \neq k}}^{n} 
    \Big(\lambda_i(0) 
    \det(\dot f_i(0), f_k(t) - f_i(0))
    \Big)
  \end{aligned}
  \end{equation}
  and
  \begin{equation*}
    -\int_0^t 
    \mu_{k - 1}(u) \det(\DD f_{k - 1}(u), f_k(t) - f_{k - 1}(u))\, \dd u
    +
    \int_0^t 
    \mu_k(u) \det(\DD f_k(u), f_k(t) - f_{k + 1}(u))\, \dd u.
  \end{equation*}
  Since 
  $\det(\DD f_{k - 1}(u), f_k(t) - f_{k - 1}(u))
  = 
  \det(\DD f_{k - 1}(u), f_k(t) - f_k(u))$
  and 
  $\det(\DD f_k(u), f_k(t) - f_{k + 1}(u))
  =
  \det(\DD f_k(u), f_k(t) - f_k(u))$,
  the integral term becomes
  \begin{equation*}
    -\int_0^t 
    \det(\mu_{k - 1}(u) \DD f_{k - 1}(u) - \mu_k(u) \DD f_k(u), 
    f_k(t) - f_k(u))\, \dd u.
  \end{equation*}
  Using the stressability condition in Equation~\eqref{eq:sdequilibrium},
  we obtain 
  \begin{equation*}
    -\int_0^t 
    \det\bigg(\frac{\partial (\lambda_k(u) \dot f_k(u))}{\partial u}, 
    f_k(t) - f_k(u)\bigg)\, \dd u.
  \end{equation*}
  Integration by parts yields
  \begin{equation}
  \label{eq:integralterm}
    \lambda_k(0)\det(\dot f_k(0),f_k(t)-f_k(0)).
  \end{equation}
  Adding terms~\eqref{eq:affinesum} and~\eqref{eq:integralterm}, we obtain
  \begin{equation*}
    H^f_\gamma(f(k, t)) - H^g_\gamma(g(n - k, t))
    = 
    \sum_{\substack{i = 0}}^{n} 
    \Big(\lambda_i(0) 
    \det(\dot f_i(0), f_k(t) - f_i(0))
    \Big),
  \end{equation*}
  which is an affine function of~$f_k(t)$. 
\end{proof}

The following Gra\ss mann-Pl{\"u}cker relation will be needed in the proof
of Theorem~\ref{thm:liftability} and can be found, e.g.,
in~\cite[Th.~6.2]{richter-gebert-2011}.

\begin{lemma}[Gra\ss mann-Pl{\"u}cker relation]
  \label{lem:bracket}
  Any four vectors $a, b, c, d \in \R^2$ satisfy
  \begin{equation*}
    \det(a, b) \det(c, d)
    -
    \det(a, c) \det(b, d)
    +
    \det(a, d) \det(b, c)
    = 0.
  \end{equation*}
\end{lemma}

The following theorem is the semi-discrete analog of the well known
Maxwell-Cremona lifting property.

\begin{theorem}
  \label{thm:liftability}
  The following two complementary statements relate semi-discrete conjugate surfaces (Definition~\ref{def:conjugate}) and semi-discrete liftings of self-stressed frameworks (Definition~\ref{def:lifting}).
  \begin{enumerate}
    \item\label{itm:lift1} 
      Let $(\lambda, \mu)$ be a self-stress for a framework $f$. Then its
      semi-discrete lifting $L(U)$ is a semi-discrete conjugate surface in
      $\R^3$.
    \item\label{itm:lift2} The regular orthogonal projection $f$ of a
      semi-discrete conjugate surface $F$ in $\R^3$ to the plane $\R^2$ is
      a stressable framework. The stress functions are given by
      Equations~\eqref{eq:stresslambda} and~\eqref{eq:stressmu} below.
  \end{enumerate}
\end{theorem}
\begin{proof}
  For~\ref{itm:lift1} we have to show linear dependence of
  \begin{equation}
    \label{eq:threevectors}
    \DD L_{k - 1}(t), \dot L_{k - 1}(t), \dot L_k(t), 
  \end{equation}
  for all $k$ and $t$ (cf.\ Remark~\ref{rem:conjugate}).
  Let $\gamma$ and $\delta$ be two increasing semi-discrete paths, ending in
  $(k, t)$ and $(k - 1, t)$, respectively, with
  \begin{equation*}
    \gamma_0 = \ldots = \gamma_{k - 1} = 0,
    \
    \gamma_k = \gamma_{k + 1} = t
    \qquad\text{and}\qquad
    \delta_0 = \ldots = \delta_{k - 1} = 0,
    \
    \delta_k = t.
  \end{equation*}
  We therefore have 
  $L_{k - 1}(t) = (f_{k - 1}, H_\delta(f(k - 1, t)))$
  and
  $L_k(t) = (f_k, H_\gamma(f(k, t)))$.
  For the sake of brevity and clarity we will omit the parameter~$0$, for example,
  $\lambda_i$ instead of $\lambda_i(0)$, for the remainder of this proof.
  We have
  \begin{align*}
    H_\gamma(f(k, t))
    =&
    \sum_{i = 0}^{k - 2} 
    \Big(\lambda_i \det(\dot f_i, f_k(t) - f_i) \Big)
    +
    \lambda_{k - 1}(t) \det(\dot f_{k - 1}(t), f_k(t) - f_{k - 1}(t))
    \\
    &
    -
    \int_0^t 
    \mu_{k - 2}(u) \det(\DD f_{k - 2}(u), f_k(t) - f_{k - 2}(u))\, \dd u
  \end{align*}
  and 
  \begin{equation}
    \begin{aligned}
      \label{eq:hdelta}
      H_\delta(f(k - 1, t))
      =&
      \sum_{i = 0}^{k - 2} 
      \Big(\lambda_i \det(\dot f_i, f_{k - 1}(t) - f_i) \Big)
      \\
      &-
      \int_0^t 
      \mu_{k - 2}(u) 
      \det(\DD f_{k - 2}(u), f_{k - 1}(t) - f_{k - 2}(u))\, \dd u.
    \end{aligned}
  \end{equation}
  Consequently,
  \begin{equation}
    \label{eq:ddf}
    \begin{aligned}
    &\DD L_{k - 1}(t) 
    =
    (\DD f_{k - 1}(t), 
    H_\gamma(f(k, t)) - H_\delta(f(k - 1, t))\\
    &=
    \Big(
    \DD f_{k - 1}(t), 
    \det\big(
    \sum_{i = 0}^{k - 2} 
    \lambda_i \dot f_i + \lambda_{k - 1}(t) \dot f_{k - 1}(t)
    - 
    \int_0^t \mu_{k - 2}(u) \DD f_{k - 2}(u)\, \dd u,
    \DD f_{k - 1}(t)\big) 
    \Big).
  \end{aligned}
  \end{equation}
  Using Leibniz' integral rule we differentiate $H_\delta$
  from Equation~\eqref{eq:hdelta} with respect to~$t$ and obtain
  \begin{equation*}
    \dot L_{k - 1}(t)
    =
    \Big(
    \dot f_{k - 1}(t),
    \det\big(
    \sum_{i = 0}^{k - 2} 
    \lambda_i \dot f_i 
    - 
    \int_0^t \mu_{k - 2}(u) \DD f_{k - 2}(u)\, \dd u,
    \dot f_{k - 1}(t)\big) 
    \Big).
  \end{equation*}
  Using the determinant rule $\det(a,b)=\det(a+\lambda b,b)$, we can rewrite this as
  \begin{equation}
    \label{eq:fk-1}
    \dot L_{k - 1}(t)
    =
    \Big(
    \dot f_{k - 1}(t),
    \det\big(
    \sum_{i = 0}^{k - 2} 
    \lambda_i\dot f_i 
    +
    \lambda_{k-1}(t)\dot f_{k-1}(t)
    - 
    \int_0^t \mu_{k-2}(u) \DD f_{k-2}(u)\, \dd u,
    \dot f_{k-1}(t)\big) 
    \Big).
  \end{equation}
  Analogously, we compute
  \begin{equation*}
    \dot L_k(t)
    =
    \Big(
    \dot f_k(t),
    \det\big(
    \sum_{i = 0}^{k - 1} 
    \lambda_i \dot f_i 
    - 
    \int_0^t \mu_{k - 1}(u) \DD f_{k - 1}(u)\, \dd u,
    \dot f_k(t)\big) 
    \Big).
  \end{equation*}
  Using Equation~\eqref{eq:sdequilibrium} for self-stresses we obtain
  \begin{equation*}
    \dot L_k(t)
    =
    \Big(
    \dot f_k(t),
    \det\big(
    \sum_{i = 0}^{k - 1} 
    \lambda_i \dot f_i 
    - 
    \int_0^t \mu_{k - 2}(u) \DD f_{k - 2}(u)
    -
    \frac{\partial}{\partial u} 
    \big(\lambda_{k - 1}(u) \dot f_{k - 1}(u)\big)\, \dd u,
    \dot f_k(t)\big) 
    \Big),
  \end{equation*}
  hence
  \begin{equation}
    \label{eq:fk}
    \dot L_k(t)
    =
    \Big(
    \dot f_k(t),
    \det\big(
    \sum_{i = 0}^{k - 2} 
    \lambda_i \dot f_i 
    +
    \lambda_{k - 1}(t) \dot f_{k - 1}(t)
    - 
    \int_0^t \mu_{k - 2}(u) \DD f_{k - 2}(u)\, \dd u,
    \dot f_k(t)\big) 
    \Big).
  \end{equation}
  Note that the right hand sides of Equations~\eqref{eq:ddf},~\eqref{eq:fk-1} and~\eqref{eq:fk} are vectors
  of the form
  \begin{equation*}
    A = (a, \det(q, a)),\quad
    B = (b, \det(q, b)),\quad
    C = (c, \det(q, c))
  \end{equation*}
  for some two-dimensional vectors $a, b, c, q$.
  We can assume without loss of generality that $c$ is a linear combination of $a, b$ which implies that $C$
  is linear combination $A, B$. Therefore $A, B, C$ are linearly dependent
  which implies linear dependence in~\eqref{eq:threevectors}, which is
  what we wanted to show for~\ref{itm:lift1}.

  For~\ref{itm:lift2}, 
  the assumption of $F$ being conjugate implies that 
  (cf.\ Definition~\ref{def:conjugate} and Remark~\ref{rem:conjugate})
  \begin{equation}
  \label{eq:conjugate}
    \det(\DD F_i, \dot F_i, \dot F_{i + 1}) = 0
    \qquad\text{and}\qquad
    \det(\DD F_{i - 1}, \dot F_{i - 1}, \dot F_i) = 0
  \end{equation}
  for the two strips adjacent to the curve $F_i$. 
  Let us denote $F_i(t)=(f_i(t),z_i(t))$ with 
  $f_i:U\to\R^2$ and $z_i :U\to\R$.
  Using this notation, Equations~\eqref{eq:conjugate} become
  \begin{equation}
    \label{eq:devcond1}
    \det(\dot f_i, \dot f_{i + 1}) \DD z_i 
    -
    \det(\DD f_i, \dot f_{i + 1}) \dot z_i 
    +
    \det(\DD f_i, \dot f_i) \dot z_{i + 1}
    = 0
  \end{equation}
  and
  \begin{equation}
    \label{eq:devcond2}
    \det(\dot f_{i - 1}, \dot f_i) \DD z_{i - 1} 
    -
    \det(\DD f_{i - 1}, \dot f_i) \dot z_{i - 1} 
    +
    \det(\DD f_{i - 1}, \dot f_{i - 1}) \dot z_i
    = 0.
  \end{equation}

  Let us define a pair of stress functions $(\lambda, \mu)$ by
  \begin{equation}
    \label{eq:stresslambda}
    \lambda_i 
    = 
    \frac{\det(\DD f_{i - 1}, \DD f_i) \dot z_i}{\det(\DD f_i, \dot f_i)
    \det(\DD f_{i - 1}, \dot f_i)}
    +
    \frac{\DD z_{i - 1}}{\det(\DD f_{i - 1}, \dot f_i)}
    -
    \frac{\DD z_i}{\det(\DD f_i, \dot f_i)}
  \end{equation}
  and
  \begin{equation}
    \label{eq:stressmu}
    \mu_i 
    =
    \frac{\det(\DD f_i, \ddot f_i) \dot z_i 
    - \det(\DD f_i, \dot f_i) \ddot z_i 
    - \det(\dot f_i, \ddot f_i) \DD z_i 
    }{\det(\DD f_i, \dot f_i)^2}.
  \end{equation}
  (Functions~$\lambda_i$ and~$\mu_i$ depend on $t$. We omit the parameter~$t$ for the sake of brevity.)
  We will show below that these stress functions constitute a
  self-stress for the orthogonal projection $f$ of $F$.
  But first we introduce the following notation (which we will also us later in
  Example~\ref{ex:1-strip}) to shorten our computations:
  \begin{equation}
    \label{eq:notation}
    \begin{array}{l@{\qquad}l@{\qquad}l}
      A = \det(\DD f_{i - 1}, \DD f_i)
      &
      D = \det(\DD f_i, \dot f_{i - 1})
      &
      I = \det(\dot f_i, \dot f_{i + 1})
      \\
      B = \det(\DD f_i, \dot f_i)
      &
      E = \det(\DD f_{i - 1}, \dot f_{i + 1})
      &
      \bar I = \det(\dot f_{i - 1}, \dot f_i)
      \\
      \bar B = \det(\DD f_{i - 1}, \dot f_{i - 1})
      &
      G = \det(\DD f_i, \ddot f_i)
      &
      J = \det(\dot f_i, \ddot f_i)
      \\
      C = \det(\DD f_i, \dot f_{i + 1})
      &
      \bar G = \det(\DD f_{i - 1}, \ddot f_{i - 1})
      &
      \bar J = \det(\dot f_{i - 1}, \ddot f_{i - 1})
      \\
      \bar C = \det(\DD f_{i - 1}, \dot f_i)
      &
      H = \det(\DD f_{i - 1}, \ddot f_i)
      &
      K = \det(\dot f_{i - 1}, \ddot f_i)
      \\
      &
      &
      L = \det(\ddot f_{i - 1}, \dot f_i).
    \end{array}
  \end{equation}
  Furthermore, we will need 
  \begin{equation*}
    \begin{array}{l}
      \dot A = 
      \det(\DD \dot f_{i - 1}, \DD f_i) + \det(\DD f_{i - 1}, \DD \dot f_i)
      = -B + D + E - \bar C
      \\
      \dot B = 
      \det(\DD \dot f_i, \dot f_i) + \det(\DD f_i, \ddot f_i)
      = -I + G
      \\
      \dot {\bar B} = 
      \det(\DD \dot f_{i - 1}, \dot f_{i - 1}) 
      + \det(\DD f_{i - 1}, \ddot f_{i - 1})
      = - \bar I + \bar G
      \\
      \dot {\bar C} = 
      \det(\DD \dot f_{i - 1}, \dot f_i) + \det(\DD f_{i - 1}, \ddot f_i)
      = -\bar I + H
      \\
      \dot {\bar I} = 
      \det(\ddot f_{i - 1}, \dot f_i) + \det(\dot f_{i - 1}, \ddot f_i)
      = L + K.
    \end{array}
  \end{equation*}

  Then the developability conditions of the strips~\eqref{eq:devcond1}
  and~\eqref{eq:devcond2} become
  \begin{equation}
    \label{eq:devcond}
    I \DD z_i - C \dot z_i + B \dot z_{i + 1} = 0
    \qquad\text{and}\qquad
    \bar I \DD z_{i - 1} - \bar C \dot z_{i - 1} + \bar B \dot z_i = 0.
  \end{equation}
  The derivative of the second equation in~\eqref{eq:devcond} becomes
  \begin{equation*}
    \dot {\bar I} \DD z_{i - 1}
    +
    \bar I \DD \dot z_{i - 1}
    -
    \dot {\bar C} \dot z_{i - 1}
    -
    \bar C \ddot z_{i - 1}
    +
    \dot {\bar B} \dot z_i
    +
    \bar B \ddot z_i
    = 0.
  \end{equation*}
  Substituting the above expressions for the derivatives and collecting
  coefficients yields
  \begin{equation}
    \label{eq:secondderiv}
    (K + L) \DD z_{i - 1}
    -
    H \dot z_{i - 1}
    +
    \bar G \dot z_i
    -
    \bar C \ddot z_{i - 1}
    +
    \bar B \ddot z_i
    = 0.
  \end{equation}

  Stress functions $\lambda_i$, $\mu_i$ and $\mu_{i - 1}$ become
  \begin{equation*}
    \lambda_i 
    = 
    \frac{A \dot z_i}{B \bar C} 
    + 
    \frac{\DD z_{i - 1}}{\bar C}
    -
    \frac{\DD z_i}{B},
    \quad
    \mu_i 
    = 
    \frac{G \dot z_i - B \ddot z_i - J \DD z_i}{B^2}
    \quad\text{and}\quad
    \mu_{i - 1} 
    = 
    \frac{\bar G \dot z_{i - 1} - \bar B \ddot z_{i - 1} - \bar J \DD z_{i
    - 1}}{\bar B^2},
  \end{equation*}
  and the derivative of $\lambda_i$ becomes
  \begin{equation*}
    \dot \lambda_i 
    =
    \frac{(\dot A \dot z_i + A \ddot z_i) B \bar C 
      - A \dot z_i (\dot B \bar C + B \dot {\bar C})}{B^2 \bar C^2}
    +
    \frac{\bar C \DD \dot z_{i - 1} - \dot {\bar C} \DD z_{i - 1}}{\bar
    C^2}
    -
    \frac{B \DD \dot z_i - \dot B \DD z_i}{\bar B^2}.
  \end{equation*}
  Substituting the above expressions for the derivatives and collecting
  coefficients yields
  \begin{equation*}
    \dot \lambda_i 
    = \Big(
    \frac{D + E}{B \bar C}
    +
    \frac{A (I - G)}{B^2 \bar C}
    +
    \frac{A (\bar I - H)}{B \bar C^2}
    \Big) \dot z_i
    -
    \frac{\dot z_{i - 1}}{\bar C}
    -
    \frac{\dot z_{i + 1}}{\bar B}
    +
    \frac{\bar I - H}{\bar C^2} \DD z_{i - 1}
    +
    \frac{G - I}{B^2} \DD z_i
    +
    \frac{A}{B \bar C} \ddot z_i.
  \end{equation*}
  Substituting $\dot z_{i+1}$ and $\dot z_{i-1}$ from Equations~\eqref{eq:devcond} yields
  \begin{equation*}
    \dot \lambda_i 
    = \Big(
    \frac{D + E}{B \bar C}
    +
    \frac{A (I - G)}{B^2 \bar C}
    +
    \frac{A (\bar I - H)}{B \bar C^2}
    -
    \frac{C}{B^2}
    -
    \frac{\bar B}{\bar C^2}
    \Big) \dot z_i
    -
    \frac{H}{\bar C^2} \DD z_{i - 1}
    +
    \frac{G}{B^2} \DD z_i
    +
    \frac{A}{B \bar C} \ddot z_i.
  \end{equation*}

  To show that the pair of stress functions $(\lambda, \mu)$ defined above
  is a self-stress we must verify Equation~\eqref{eq:sdequilibrium}, i.e.,
  we must show $e_i = 0$ where
  \begin{equation}
    \label{eq:ei}
    e_i = \dot \lambda_i \dot f_i + \lambda_i \ddot f_i + \mu_i \DD f_i 
    - \mu_{i - 1} \DD f_{i - 1}.
  \end{equation}

  Since we assume that the projected framework $f$ is regular we have that
  at all points the smooth and discrete derivative 
  $\dot f_i, \DD f_{i - 1}$ are linearly independent (cf.\
  Equation~\eqref{eq:regular}).
  Consequently, any vector $a \in \R^2$ with $\det(\dot f_i, a) = \det(\DD
  f_{i - 1}, a) = 0$ is the zero vector.
  Hence it will be sufficient to show
  $\det(\dot f_i,e_i)= \det(\DD f_{i-1},e_i)=0$.
  Let us first compute $\det(\DD f_{i - 1}, e_i)$:
  \medskip

  \mmeq{}{\det(\DD f_{i - 1}, e_i)
  \overset{\eqref{eq:ei}}{=}
    \dot \lambda_i \det(\DD f_{i - 1}, \dot f_i) 
    + \lambda_i \det(\DD f_{i - 1}, \ddot f_i)
    + \mu_i \det(\DD f_{i - 1}, \DD f_i)
  }\medskip

  \mmeq{=}{
  \dot \lambda_i \bar C + \lambda_i H + \mu_i A
  }\medskip

  \mmeq{=}{
    \Big(
    \frac{D + E}{B}
    +
    \frac{A (I - G)}{B^2}
    +
    \frac{A (\bar I - H)}{B \bar C}
    -
    \frac{C \bar C}{B^2}
    -
    \frac{\bar B}{\bar C}
    \Big) \dot z_i
    -
    \frac{H}{\bar C} \DD z_{i - 1}
    +
    \frac{G \bar C}{B^2} \DD z_i
    +
    \frac{A}{B} \ddot z_i
  }\medskip

  \mmeq{}{
    +
    \Big(
    \frac{A H}{B \bar C} \dot z_i
    + 
    \frac{H}{\bar C} \DD z_{i - 1}
    -
    \frac{H}{B} \DD z_i
    \Big)
    +
    \Big(
    \frac{A G}{B^2} \dot z_i - \frac{A}{B} \ddot z_i 
    - \frac{A J}{B^2} \DD z_i
    \Big)
  }\medskip

  \mmeq{=}{
    \Big(
    \frac{D + E}{B}
    +
    \frac{A I - C \bar C}{B^2}
    +
    \frac{A \bar I}{B \bar C}
    -
    \frac{\bar B}{\bar C}
    \Big) \dot z_i
    +
    \Big(
    \frac{\bar C G - A J}{B^2}
    -
    \frac{H}{B}
    \Big) \DD z_i.
  }\medskip

  \noindent
  We will use the bracket notation for the determinant $[\cdot,\cdot]=
  \det(\cdot,\cdot)$.
  The numerator of the second parenthesis over the common denominator is
  \begin{equation}
    \label{eq:numerator}
    A J - \bar C G + B H
    =
    [\DD f_{i - 1}, \DD f_i]
    [\dot f_i, \ddot f_i]
    -
    [\DD f_{i - 1}, \dot f_i]
    [\DD f_i, \ddot f_i]
    +
    [\DD f_i, \dot f_i]
    [\DD f_{i - 1}, \ddot f_i]
    = 0,
  \end{equation}
  which follows from Lemma~\ref{lem:bracket}.
  The numerator of the first parenthesis over the common denominator is
  \medskip

  \mmeq{}{B (A \bar I - B \bar B + \bar C D)
  +
  \bar C (A I - C \bar C + B E)
  }\medskip

  \mmeq{=}{
    B (
      [\DD f_{i - 1}, \DD f_i]
      [\dot f_{i - 1}, \dot f_i]
      -
      [\DD f_i, \dot f_i]
      [\DD f_{i - 1}, \dot f_{i - 1}]
      +
      [\DD f_{i - 1}, \dot f_i]
      [\DD f_i, \dot f_{i - 1}]
      )
  }\medskip

  \mmeq{}{
      +\
   \bar C (
      [\DD f_{i - 1}, \DD f_i]
      [\dot f_i, \dot f_{i + 1}]
      -
      [\DD f_i, \dot f_{i + 1}]
      [\DD f_{i - 1}, \dot f_i]
      +
      [\DD f_i, \dot f_i]
      [\DD f_{i - 1}, \dot f_{i + 1}]
      ) 
      = 0,
  }\medskip

  \noindent
  which follows from Lemma~\ref{lem:bracket}.
  Consequently, the two parentheses
  vanish and we obtain 
  $\det(\DD f_{i - 1}, e_i) = 0$.
  Let us now compute $\det(\dot f_i, e_i)$:
  \medskip

  \mmeq{}{\det(\dot f_i, e_i)
  \overset{\eqref{eq:ei}}{=}
  \lambda_i \det(\dot f_i, \ddot f_i)
  +
  \mu_i \det(\dot f_i, \DD f_i)
  -
  \mu_{i - 1} \det(\dot f_i, \DD f_{i - 1})
  = 
  \lambda_i J - \mu_i B + \mu_{i - 1} \bar C
  }\medskip

  \mmeq{=}{
    \Big(
    \frac{A J}{B \bar C} \dot z_i
    + 
    \frac{J}{\bar C} \DD z_{i - 1}
    -
    \frac{J}{B} \DD z_i
    \Big) 
    +
    \Big(
    -\frac{G}{B} \dot z_i
    +
    \ddot z_i
    +
    \frac{J}{B} \DD z_i
    \Big) 
    +
    \Big(
    \frac{\bar C \bar G}{\bar B^2} \dot z_{i - 1}
    -
    \frac{\bar C}{\bar B} \ddot z_{i - 1}
    -
    \frac{\bar C \bar J}{\bar B^2} \DD z_{i - 1}
    \Big) 
  }\medskip

  \mmeq{=}{
    \frac{\bar C \bar G}{\bar B^2}
    \dot z_{i - 1}
    +
    \Big(
    \frac{A J}{B \bar C} - \frac{G}{B} 
    \Big) \dot z_i
    +
    \Big(
    \frac{J}{\bar C} - \frac{\bar C \bar J}{\bar B^2} 
    \Big) \DD z_{i - 1}
    + 
    \ddot z_i
    -
    \frac{\bar C}{\bar B} \ddot z_{i - 1}
  }\medskip

  \mmeq{\overset{\eqref{eq:secondderiv}}{=}}{
    \Big(
    \frac{\bar C \bar G}{\bar B^2} + \frac{H}{\bar B}
    \Big)
    \dot z_{i - 1}
    +
    \Big(
    \frac{A J}{B \bar C} - \frac{G}{B} - \frac{\bar G}{\bar B} 
    \Big) \dot z_i
    +
    \Big(
    \frac{J}{\bar C} - \frac{\bar C \bar J}{\bar B^2} - \frac{K + L}{\bar B} 
    \Big) \DD z_{i - 1}
  }\medskip

  \mmeq{\overset{\eqref{eq:devcond}}{=}}{
    \Big(
    \frac{A J}{B \bar C} - \frac{G}{B} + \frac{H}{\bar C} 
    \Big) \dot z_i
    +
    \Big(
    \frac{J}{\bar C} - \frac{\bar C \bar J - \bar G \bar I}{\bar B^2} 
    - \frac{K + L}{\bar B} + \frac{H \bar I}{\bar B \bar C}
    \Big) \DD z_{i - 1}.
  }\medskip

  \noindent
  The numerator of the first parenthesis over the common denominator is 
  $AJ-\bar CG+BH$ and therefore equals the one in
  Equation~\eqref{eq:numerator}, hence vanishes.
  The numerator of the second parenthesis over the common denominator is 
  \medskip

  \mmeq{}{
    \bar B (\bar B J - \bar C K + H \bar I)
    -
    \bar C (\bar B L - \bar G \bar I + \bar C \bar J)
  }\medskip

  \mmeq{=}{
    \bar B (
      [\DD f_{i - 1}, \dot f_{i - 1}]
      [\dot f_i, \ddot f_i]
      -
      [\DD f_{i - 1}, \dot f_i]
      [\dot f_{i - 1}, \ddot f_i]
      +
      [\DD f_{i - 1}, \ddot f_i]
      [\dot f_{i - 1}, \dot f_i]
    )
  }\medskip

  \mmeq{}{
    -\ 
    \bar C (
      [\DD f_{i - 1}, \dot f_{i - 1}]
      [\ddot f_{i - 1}, \dot f_i]
      -
      [\DD f_{i - 1}, \ddot f_{i - 1}]
      [\dot f_{i - 1}, \dot f_i]
      +
      [\DD f_{i - 1}, \dot f_i]
      [\dot f_{i - 1}, \ddot f_{i - 1}]
    ) 
    = 0,
  }\medskip

  \noindent
  which follows from Lemma~\ref{lem:bracket}. Hence, the two parentheses vanish and
  we obtain $\det(\bar f_i, e_i) = 0$. Consequently, $e_i = 0$ which is
  equivalent to Equation~\eqref{eq:sdequilibrium}.
\end{proof}

Theorem~\ref{thm:liftability} implies that a framework can be lifted to a semi-discrete
conjugate surface if and only if it is stressable. We therefore make the
following definition.

\begin{definition}
  We call a semi-discrete framework $f$ \emph{liftable} if there exists a
  semi-discrete conjugate surface $F = (f, z)$ in $\R^3$ which
  orthogonally projects to $f$.
\end{definition}

Theorem~\ref{thm:liftability} implies that liftability and stressability
are equivalent.

\begin{SCfigure}[2][h]
  \begin{overpic}[width=.6\textwidth]{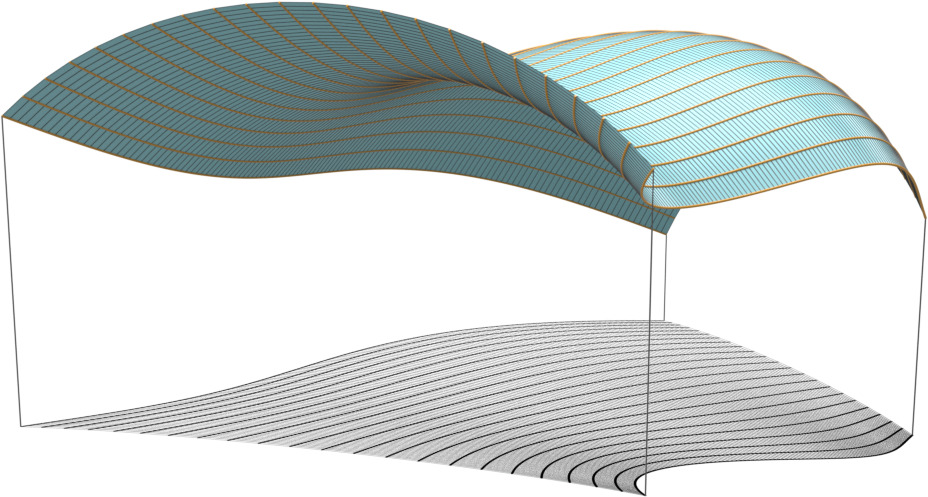}
    \put(34,1){\small$f$}
    \put(85,30){\small$F$}
  \end{overpic}
  \caption{Lifting of a stressable framwork. The semi-discrete framework
  $f$ in the plane is the orthogonal projection of a semi-discrete
  conjugate surface $F$. All strips on $F$ are developable surfaces. The
  smooth curves $F_i$ on $F$ project to the smooth curves $f_i$ in the
  framework $f$. The rulings of each strip of $F$ map to the rulings of
  the framework $f$, i.e., the orthogonal projection of $\DD F_i(t)$ is
  $\DD f_i(t)$ for all $t$.}
  \label{fig:lifting}
\end{SCfigure}

\section{Stressability of Special Semi-Discrete Frameworks}
\label{sec:stressability}

Suppose we are given a framework in the plane $f : U \to \R^2$.
Its stressability is equivalent to the existence of stress functions
$(\lambda, \mu)$ which fulfill Equation~\eqref{eq:sdequilibrium}:
\begin{equation*}
  \dot \lambda_i \dot f_i + \lambda_i \ddot f_i + \mu_i \DD f_i 
  - \mu_{i - 1} \DD f_{i - 1} = 0.
\end{equation*}

For any given framework $f$ the above equation is a linear system of ODEs
in $\lambda$ and $\mu$, so in non-degenerate cases (i.e., if the functions are sufficiently smooth and do not vanish) the system is locally solvable, therefore the framework is liftable. 

As parameter space we are always using $U = \{0, \ldots, n\} \times [0,
T]$. In that way we are putting an emphasis on considering $n - 1$ strips
between $0$ and $n$. The boundary strips between $f_{-1}, f_0$ and between
$f_n, f_{n + 1}$ are introduced for an easier handling of potential
boundary forces (see Definition~\ref{def:sdframework} and
Figure~\ref{fig:path} left for an illustration of the semi-discrete
parameter domain with the two boundary strips).

\begin{example}[1-curve framework]
  \label{ex:1-curve}
  The most minimalist non-trivial example is a framework that consists of
  only one curve with boundary force vectors attached along it (see
  Figure~\ref{fig:path} left).
  In this case $U=\{0\}\times[0,T]$.
  The self-stressability condition of this framework reads
  \begin{equation*}
  \dot \lambda_0 \dot f_0 + \lambda_0 \ddot f_0 + \mu_0 \DD f_0 
  - \mu_{-1} \DD f_{-1} = 0.
  \end{equation*}
  We can prescribe the curves $f_{-1}, f_0, f_1$ and one of the three
  stress functions $\lambda_0, \mu_{-1}, \mu_0$.
  We need $f_{-1}$ and $f_1$ only to prescribe the directions of the
  boundary forces. Then the remaining two stress functions can be
  determined in such a way that the condition above is satisfied.
\end{example}

In the case where we have no boundary forces, i.e., $\mu_{-1} = \mu_n = 0$,
Equation~\eqref{eq:Hk} implies $H_\gamma(f(0, t)) = 0$.
Hence, the boundary curve $L(0, \cdot)$ of the lifting $L$ lies in a
horizontal plane. 
Recall that Proposition~\ref{prop:affine} implies that the lifting of
the first curve is an affine image of the lifting of the last curve if we
lift the framework in the reversed order.
Consequently, the other
boundary curve, $L(n, \cdot)$, also lies in a plane if the boundary forces are
vanishing. We obtain the following proposition.

\begin{proposition}
  Stressed frameworks $f$ with vanishing boundary forces $($$\mu_{-1} =
  \mu_n = 0$$)$ have liftings with planar boundary curves.
\end{proposition}

Note that the reverse statement does not hold as can be deduced from
simple counterexamples. 

The above condition of vanishing boundary forces ($\mu_{-1} = \mu_n = 0$) implies a condition for the framework $f$ to be liftable.
Let us determine this condition for the simplest (non-trivial) case of a
framework consisting of just one strip.

\begin{figure}[htb]
  \begin{overpic}[width=.4\textwidth]{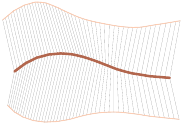}
    \put(67,56){\small$f_1$}
    \put(96,24){\small$f_0$}
    \put(67,2){\small$f_{-1}$}
  \end{overpic}
  \hfill
  \begin{overpic}[width=.53\textwidth]{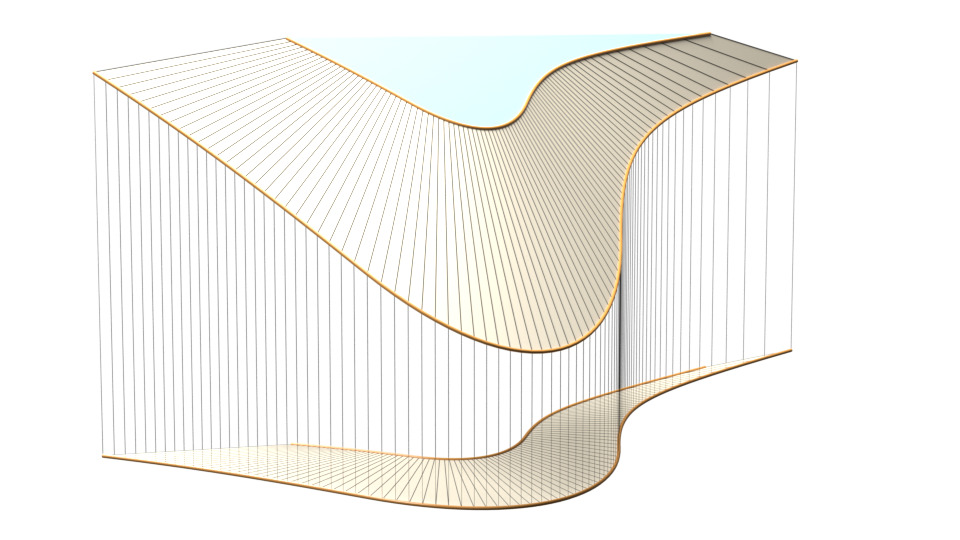}
  \end{overpic}
  \caption{\emph{Left:} A curve $f_0$ with boundary forces pointing from
  $f_0$ to the curves $f_{-1}$ and $f_1$. Any such configuration of a curve
  with two neighboring curves is stressable (cf.\
  Example~\ref{ex:1-curve}).
  \emph{Right:} A framework $f$ consisting of just one strip
  (i.e., two neighboring curves) and with no boundary forces ($\mu_{-1}=\mu_1=0$). Such a framework is liftable only if Equation~\eqref{eq:1-strip-liftability} holds (cf.\
  Example~\ref{ex:1-strip}).}
  \label{fig:1-strip}
\end{figure}

\begin{example}[1-strip framework]
  \label{ex:1-strip}
  Let $f:U\to\R^2$, $U=\{0,1\}\times[0, T]$,
  be a framework consisting of just one strip (i.e.,
  two neighboring curves) and with no boundary forces ($\mu_{-1} = \mu_1 = 0$), see Figure~\ref{fig:1-strip} right.
  Such a framework is liftable if there are
  functions $\lambda_0, \lambda_1, \mu_0$ satisfying the system
  \begin{align*}
  \left\{
  \begin{array}{l}
    \dot \lambda_0 \dot f_0 + \lambda_0 \ddot f_0 + \mu_0 \DD f_0 = 0
    \\
    \dot \lambda_1 \dot f_1 + \lambda_1 \ddot f_1 - \mu_0 \DD f_0 = 0.
   \end{array}
   \right.
  \end{align*}
  By scalar multiplying both equations with two linearly independent
  vectors $\dot f_0^\perp, \DD f_0^\perp$ we obtain an equivalent system of
  four equations (${}^\perp$ denotes the rotation through 90 degrees)
  \begin{align*}
  \left\{
  \begin{array}{l}
    \lambda_0 \det(\dot f_0, \ddot f_0) + \mu_0 \det(\dot f_0, \DD f_0) = 0
    \\
    \dot \lambda_1 \det(\dot f_0, \dot f_1) 
    + \lambda_1 \det(\dot f_0, \ddot f_1) - \mu_0 \det(\dot f_0, \DD f_0) = 0
    \\
    \dot \lambda_0 \det(\DD f_0, \dot f_0) + \lambda_0 \det(\DD f_0, \ddot
    f_0) = 0
    \\
    \dot \lambda_1 \det(\DD f_0, \dot f_1) + \lambda_1 \det(\DD f_0, \ddot
    f_1) = 0.
    \end{array}
   \right.
  \end{align*}
  Using the notation introduced in Equations~\eqref{eq:notation} with $i =
  1$ the system becomes
  \begin{align}
    \label{eq:1-strip1}
    &\lambda_0 \bar J - \mu_0 \bar B = 0
    \\
    \label{eq:1-strip2}
    &\dot \lambda_1 \bar I + \lambda_1 K + \mu_0 \bar B = 0
    \\
    \label{eq:1-strip3}
    &\dot \lambda_0 \bar B + \lambda_0 \bar G = 0
    \\
    \label{eq:1-strip4}
    &\dot \lambda_1 \bar C + \lambda_1 H = 0.
  \end{align}
  Equations~\eqref{eq:1-strip3} and~\eqref{eq:1-strip4} imply 
  \begin{equation*}
    \lambda_0 = \exp\left(- \int \frac{\bar G}{\bar B}\right)
    \qquad\text{and}\qquad
    \lambda_1 = \exp\left(- \int \frac{H}{\bar C}\right).
  \end{equation*}
  Substituting Equation~\eqref{eq:1-strip1} into
  Equation~\eqref{eq:1-strip2} yields
  \begin{equation*}
    \dot \lambda_1 \bar I + \lambda_1 K + \lambda_0 \bar J = 0.
  \end{equation*}
  By substituting the $\lambda$'s into this equation we obtain the
  condition on the liftability of the 1-strip framework $f$:
  \begin{equation*}
    - \bar I \frac{H}{\bar C} \exp\left(-\int \frac{H}{\bar C}\right)
    + K \exp\left(-\int \frac{H}{\bar C}\right)
    + \bar J \exp\left(-\int \frac{\bar G}{\bar B}\right)
    = 0,
  \end{equation*}
  or equivalently
  \begin{equation*}
    \bar C K - H \bar I + \bar C \bar J 
    \exp\left(\int \frac{H}{\bar C} - \frac{\bar G}{\bar B}\right)
    = 0.
  \end{equation*}
  By the Grassman-Pl{\"u}cker relation (Lemma~\ref{lem:bracket}) we obtain
  \begin{equation*}
    \bar C K - H \bar I 
    =[\DD f_0, \dot f_1] [\dot f_0, \ddot f_1] 
     -[\DD f_0, \ddot f_1] [\dot f_0, \dot f_1]
    =[\DD f_0, \dot f_0] [\dot f_1, \ddot f_1]
    =\bar B J.
  \end{equation*}
  Consequently, our condition reads
  \begin{equation*}
    -\frac{\bar B J}{\bar C \bar J}
    = \exp\left(\int \frac{H}{\bar C} - \frac{\bar G}{\bar B}\right)
    = \exp(M),
    \qquad\text{where}\qquad
    M = \int \frac{H}{\bar C} - \frac{\bar G}{\bar B}.
  \end{equation*}
  Differentiating yields 
  \begin{equation*}
    -\frac{(\dot {\bar B} J + \bar B \dot J) \bar C \bar J - 
    \bar B J (\dot {\bar C} \bar J + \bar C \dot {\bar J})}{\bar C^2 \bar J^2}
    =
    \exp(M) \dot M
    =
    -\frac{\bar B J}{\bar C \bar J}
    \Big(
    \frac{H}{\bar C}
    -
    \frac{\bar G}{\bar B}
    \Big),
  \end{equation*}
  which is equivalent to
  \begin{equation*}
    J \bar J (\bar B (\dot {\bar C} + H) - \bar C (\dot {\bar B} + G))
    + \bar B \bar C (\dot J \bar J - J \dot {\bar J})
    = 0.
  \end{equation*}
  Switching back to the determinant notation we get the condition
  \begin{equation}
    \label{eq:1-strip-liftability}
    \begin{aligned}
      &[\dot f_0, \ddot f_0]
      [\dot f_1, \ddot f_1]
      \big(
      [\DD f_0, \dot f_0] (2 [\DD f_0, \ddot f_1] - [\dot f_0, \dot f_1])
      -
      [\DD f_0, \dot f_1] (2 [\DD f_1, \ddot f_1] - [\dot f_0, \dot f_1])
      \big)
      \\
      =&
      [\DD f_0, \dot f_0]
      [\DD f_0, \dot f_1]
      \big(
      [\dot f_0, \ddot f_0]
      [\dot f_1, \dddot{f_1}]
      -
      [\dot f_1, \ddot f_1]
      [\dot f_0, \dddot{f_0}]
      \big).
    \end{aligned}
  \end{equation}
\end{example}

From Example~\ref{ex:1-strip} we obtain the following condition.

\begin{proposition}
  A one-strip framework $f:U\to\R^2$, $U=\{0,1\}\times[0,T]$, with no boundary forces $($$\mu_{-1}=\mu_1=0$$)$, is liftable if
  and only if Equation~\eqref{eq:1-strip-liftability} holds.
\end{proposition}

\begin{remark}
  It remains an \emph{open problem} to find liftability conditions for a
  framework with $n$ strips with $\mu_{-1} = \mu_n = 0$ for $n \geq 2$.
\end{remark}

\section*{Acknowledgements}

C.M.\ gratefully acknowledges the support by the Austrian
Science Fund (FWF) through grant I~4868 (DOI 10.55776/I4868) and by the Heilbronn Institute for
Mathematical Research (HIMR) through an International Visitor Award, funded by the UKRI/EPSRC Additional Funding Programme for Mathematical Sciences.

\bibliographystyle{abbrv}
\bibliography{main.bib}

\bigskip 

\noindent
\footnotesize \textbf{Authors' addresses:}

\bigskip

\noindent{Department of Mathematical Sciences, University of Liverpool, UK} 
\hfill \texttt{karpenk@liverpool.ac.uk}
\medskip  

\noindent{Institute of Discrete Mathematics and Geometry, TU Wien, Austria} 
\hfill \texttt{cmueller@geometrie.tuwien.ac.at}
\medskip

\noindent{Department of Mathematical Sciences, University of Liverpool, UK} 
\hfill \texttt{annap@liverpool.ac.uk}

\end{document}